\documentclass[11pt]{article}

\usepackage{amsthm}
\usepackage{latexsym}
\usepackage{dsfont}
\usepackage{bbm}
\usepackage{amssymb}
\usepackage{amsmath}
\usepackage{graphicx}
\usepackage[noblocks,affil-it]{authblk}

\numberwithin{equation}{section}

\theoremstyle{plain}
\newtheorem{Theorem}{Theorem}[section]
\newtheorem{Lemma}[Theorem]{Lemma}

\theoremstyle{remark}
\newtheorem{Rem}[Theorem]{Remark}
\theoremstyle{definition}

\newcommand{\ts}[1]{t^{(#1)}}

\DeclareMathOperator{\N}{\mathbb{N}}

\DeclareMathOperator{\R}{\mathbb{R}}

\DeclareMathOperator{\Prob}{\mathbb{P}}
\DeclareMathOperator{\Var}{\mathrm{Var}}
\DeclareMathOperator{\E}{\mathbb{E}}

\DeclareMathOperator{\1}{\mathbbm{1}}

\title{Weak convergence of renewal shot noise processes in the case of slowly varying normalization}
\date{\today}
\author[1]{Alexander Iksanov\thanks{E-mail: iksan@univ.kiev.ua}}
\affil[1]{Faculty of Cybernetics, Taras Shevchenko National University of Kyiv, 01601 Kyiv, Ukraine}
\author[2]{Zakhar Kabluchko\thanks{E-mail: zakhar.kabluchko@uni-muenster.de}}
\affil[2]{Institut f\"{u}r Mathematische Statistik, Westf\"{a}lische Wilhelms-Universit\"{a}t M\"{u}nster, 48149 M\"{u}nster, Germany}
\author[1,2]{Alexander Marynych\thanks{E-mail: marynych@unicyb.kiev.ua}}

\begin{document}

\thispagestyle{empty}
\maketitle

\begin{abstract}
We investigate weak convergence of finite-dimensional
distributions of a renewal shot noise process $(Y(t))_{t\geq 0}$
with deterministic response function $h$ and the shots occurring
at the times $0 = S_0 < S_1 < S_2<\ldots$, where $(S_n)$ is a
random walk with i.i.d.\ jumps. There has been an
outbreak of recent activity around this topic. We are interested
in one out of few cases which remained open: $h$ is regularly varying at
$\infty$ of index $-1/2$ and the integral of $h^2$ is infinite.
Assuming that $S_1$ has a moment of order $r>2$ we use a strong
approximation argument to show that the random fluctuations of
$Y(s)$ occur on the scale $s=t+g(t,u)$ for $u\in [0,1]$, as $t\to\infty$, and, on
the level of finite-dimensional distributions, are well
approximated by the sum of a Brownian motion and a Gaussian process with independent values (the two processes being independent).
The scaling function $g$ above depends on the slowly varying
factor of $h$. If, for instance, $\lim_{t\to\infty}t^{1/2}h(t)\in (0,\infty)$, then
$g(t,u)=t^u$.
\end{abstract}



\section{Introduction}  \label{sec:intro}

Let $(\xi_k)_{k\in\N}$ be independent copies of a positive random
variable $\xi$. Define a zero-delayed standard random walk
$(S_n)_{n\in\N_0}$, where $\N_0:=\N\cup\{0\}$, by $$S_0:=0,\quad S_n:=\xi_1+\ldots+\xi_n,\quad
n\in\N.$$ For a locally bounded, measurable function $h:\R^+\to\R$, where $\R^+:=[0,\infty)$, put $Y(t):=\sum_{k\geq
0}h(t-S_k)\1_{\{S_k\leq t\}}$, $t\geq 0$. The process
$Y:=(Y(t))_{t\geq 0}$ is called {\it renewal shot noise process}.

The renewal shot noise processes and their natural generalizations
called {\it random processes with immigration} arise in most of
natural sciences as well as diverse areas of applied probability.
See \cite{Iksanov+Marynych+Meiners:2015-1} for a list of
possible applications and the definition of the latter processes.
A nice survey of earlier relevant literature can be found in
\cite{Vervaat:1979}.

Continuing the line of research initiated in
\cite{Iksanov:2013,Iksanov+Marynych+Meiners:2014,Iksanov+Marynych+Meiners:2015-1,Iksanov+Marynych+Meiners:2015-2}
 we investigate weak convergence of the
renewal shot noise processes. Here is a brief survey of the
previously known results concerning weak convergence of
finite-dimensional distributions which hold under the assumption
that $\sigma^2:=\Var\,\xi<\infty$. As for the case of
infinite variance we refer the reader to the cited papers.

If the law of $\xi$ is nonarithmetic and $h:\R^+\to\R$ is a
c\`{a}dl\`{a}g function such that $|h(t)|\wedge 1$ is directly
Riemann integrable on $\R^+$ (we write $x\wedge y$ for $\min(x,y)$;
the definition of the direct Riemann integrability can be found in
Section \ref{tech}), then the finite-dimensional distributions of
$(Y(t+u))_{u\in\R}$ converge weakly, as $t\to\infty$, to those of
a stationary shot noise process. This is a consequence of Theorem
2.2 in \cite{Iksanov+Marynych+Meiners:2015-2}. Weak convergence of
one-dimensional distributions has earlier been obtained in Theorem
2.1 in \cite{Iksanov+Marynych+Meiners:2014}.

If the law of $\xi$ is nonarithmetic and $h:\R^+\to\R$ is a
locally bounded, a.e.\ continuous, eventually nonincreasing and
non-integrable function with $\int_0^\infty h^2(y){\rm
d}y<\infty$, then $Y(t)-\mu^{-1}\int_0^t h(y){\rm d}y$ converges
in distribution, as $t\to\infty$ (see Theorem 2.4 (C1) in
\cite{Iksanov+Marynych+Meiners:2014}). Here and hereafter
$\mu:=\E\xi$. We believe that the finite-dimensional distributions
of $(Y(t+u)-\mu^{-1}\int_0^{t+u}h(y){\rm d}y)_{u\in\R}$ converge
weakly, but this has never been proved.

Let $h : \R^+ \to\R$ be locally bounded, measurable,
eventually monotone and
\begin{equation}\label{regular}
h(t)\sim t^\beta\ell(t),\quad t\to\infty
\end{equation}
for some $\beta>-1/2$ and some $\ell$ slowly varying\footnote{A
positive measurable function $L$, defined on some neighborhood of
$\infty$, is called {\it slowly varying} at $\infty$ if
$\lim_{t\to\infty} (L(ut)/L(t))=1$ for all $u> 0$.} at $\infty$.
Then
\begin{equation}\label{conv_to_int_bm}
\bigg({Y(ut)-\mu^{-1}\int_0^{ut}h(y){\rm d}y\over \sqrt{\sigma^2\mu^{-3}t} h(t)}\bigg)_{u\geq 0}~\stackrel{\mathrm{f.d.}}{\to}~
\bigg(\int_{[0,\,u]}(u-y)^\beta{\rm d}B(y)\bigg)_{u\geq 0},\quad
t\to\infty,
\end{equation}
where $\stackrel{\mathrm{f.d.}}{\to}$ denotes weak
convergence of finite-dimensional distributions, $(B(u))_{u\geq
0}$ is a Brownian motion. This follows from
Theorem 2.7 in \cite{Iksanov+Marynych+Meiners:2014} in the case
when $h$ is eventually nonincreasing and from \cite{Iksanov:2013}
in the case when $h$ is eventually nondecreasing.

In this paper we treat the borderline situation when $\beta$ in
\eqref{regular} equals $1/2$ yet the function $h^2$ is
nonintegrable. This case bears some similarity with the case
$\beta>-1/2$ (normalization is needed; the limit is Gaussian) and
is very different from the case when $h^2$ is integrable. The
principal new feature of the present case is necessity of
sublinear time scaling as opposed to the time scalings $t+u$ and
$ut$ used for the other regimes.

As might be expected of a transitional regime there are additional
technical complications. In particular, the techniques (tools
related to stationarity; the continuous mapping theorem along with the
functional limit theorem for the first-passage time process of
$(S_n)$) used for the other regimes cannot be exploited here. Our
main technical tool is a strong approximation theorem.

\begin{figure}
\begin{center}
\includegraphics[width=0.99\textwidth]{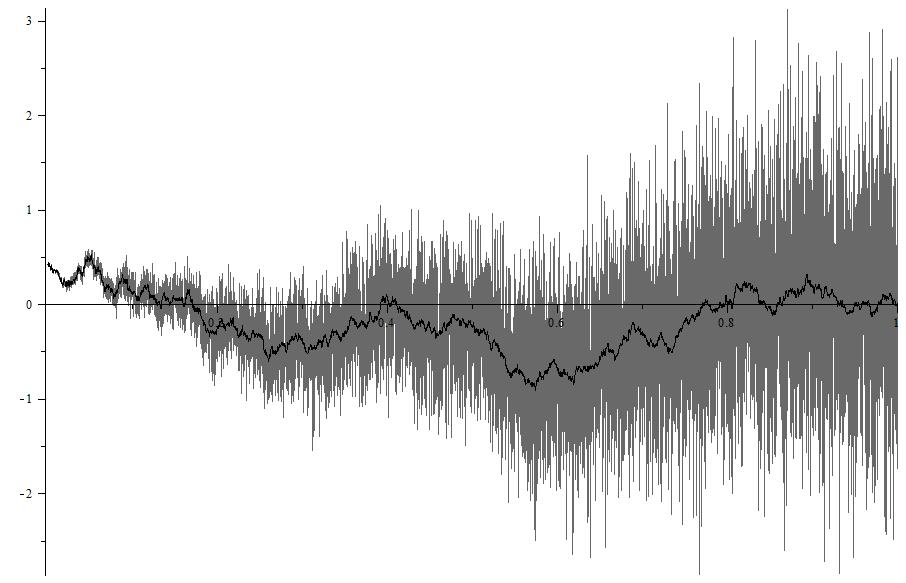}
\label{fig:process}
\caption{Grey graph: The limit process $X$. Black graph: The Brownian motion in reversed time
$B(1-u)$.}
\end{center}
\end{figure}

Now we introduce a limit process $X:=(X(u))_{u\in
[0,1]}$ appearing in Theorem \ref{main1} which is our main result. Let $B:=(B(u))_{u\in
[0,1]}$ denote a Brownian motion independent of $D:=(D(u))_{u\in
[0,1]}$, a centered Gaussian process with independent values which
satisfies $\E D^2(u)=u$. Then we set
$$X(u)=B(1-u)+D(u),\quad u\in [0,1].$$
In
different contexts such a process has arisen in recent papers
\cite{Bogachev+Su:2007,Bourgade:2010}. The presence
of $D$ makes the paths of $X$ highly irregular; see Figure~\ref{fig:process}. In particular, no
version of $X$ lives in the Skorokhod space of right-continuous
functions with finite limits from the left. The covariance
structure of $X$ is very similar to that of $B$: for any $u,v\in
[0,1]$
$${\rm cov}(X(u),X(v))=
\begin{cases}
        (1-u)\wedge(1-v), &   \text{if} \ u\neq v,   \\
        1, & \text{if} \ u=v,
\end{cases}$$
whereas ${\rm cov}(B(1-u),B(1-v))=(1-u)\wedge
(1-v)$. Among others, this shows that neither $X$, nor $X(1-\cdot)$ is a self-similar process.
\begin{Theorem}\label{main1}
Suppose that $\E\xi^r<\infty$ for some $r>2$ and that $h:\R^+\to\R$ is a
right-continuous, locally bounded and eventually nonincreasing
function. If
\begin{equation}\label{a1}
h(t)~ \sim~ t^{-1/2}\ell(t),\quad t\to\infty
\end{equation}
for some $\ell$ slowly varying at $\infty$ such that
$\int_0^\infty h^2(y){\rm d}y=\infty$, then, as $t\to\infty$,
$$\bigg({Y(t+g(t,u))-\mu^{-1}\int_0^{t+g(t,u)}h(y){\rm d}y
\over \sqrt{\sigma^2\mu^{-3}\int_0^th^2(y){\rm d}y}}\bigg)_{u\in
[0,1]} ~\stackrel{\mathrm{f.d.}}{\to}~
(X(u))_{u\in[0,1]},$$
where $\sigma^2=\Var\,\xi$,
$\mu=\E\xi$, and $g:\R^+\times [0,1]\to\R^+$ is any nondecreasing in the second coordinate function that satisfies
\begin{equation}\label{choice}
\lim_{t\to\infty}{\int_0^{g(t,u)}h^2(y){\rm d}y\over \int_0^th^2(y){\rm d}y}=u
\end{equation}
for each $u\in[0,1]$.
\end{Theorem}
\begin{Rem}
To facilitate comparison with \eqref{conv_to_int_bm}, observe that, under \eqref{regular} with $\beta>-1/2$,
$$\sqrt{\int_0^th^2(y){\rm d}y}~ \sim~
(2\beta+1)^{-1/2}\sqrt{t}h(t),\quad t\to\infty,$$ see Lemma
\ref{bgt1}(a), and therefore the normalization in \eqref{conv_to_int_bm} can be replaced (up to a multiplicative constant) by $(\int_0^th^2(y){\rm d}y)^{1/2}$.
\end{Rem}
\begin{Rem}
Set $m(t):=\int_0^th^2(y){\rm d}y$, $t>0$ and observe that, under \eqref{a1}, $m$ is a slowly varying function (see Lemma \ref{bgt1}(b)) diverging to $+\infty$. Since $m$ is nondecreasing and continuous, the generalized inverse function $m^\leftarrow$ is increasing. Putting $g(t,u)=m^\leftarrow(um(t))$ gives us a nondecreasing in $u$ function that satisfies \eqref{choice}.
\end{Rem}

\begin{Rem}\label{second}
Here we point out three types of possible time scalings which correspond to 'moderate', 'slow' and 'fast' slowly varying $\ell$ in \eqref{a1}.

\noindent {\sc 'Moderate' $\ell$}. If
\begin{equation}\label{a2}
\ell(t)=(\log t)^{(\rho-1)/2}L(\log t)
\end{equation}
for some $\rho>0$ and some $L$ slowly varying at $\infty$, then
$$m(t)=\int_0^t h^2(y){\rm d}y =\int_0^{\log t} h^2(e^y)e^y{\rm
d}y \sim \rho^{-1}(\log t)^\rho L^2(\log t),\quad t\to\infty$$
by Lemma \ref{bgt1}(a) because $h^2(e^y)e^y \sim  y^{\rho-1}L^2(y)$.
Hence, we may take $g(t,u)=t^{u^{1/\rho}}$.
\end{Rem}

\noindent {\sc 'Slow' $\ell$}. If $$\ell(t)=(\log
t)^{-1/2}(\log\log t)^{(\rho-1)/2}L(\log\log t)$$ for some
$\rho>0$ and some $L$ slowly varying at $\infty$, then $$m(t) \sim \rho^{-1}(\log\log t)^\rho L^2(\log\log t)$$ (which can be checked as in the 'moderate' case) and one may take
$g(t,u)=\exp((\log t)^{u^{1/\rho}})$.

\noindent {\sc 'Fast' $\ell$}. If $$\ell(t)=\exp((\rho/2)(\log t)^\gamma)(\log t)^{(\gamma-1)/2}L(\exp((\log t)^\gamma))$$ for some $\rho>0$, $\gamma\in(0,1)$ and some $L$ slowly varying at $\infty$, then $$m(t)~\sim~(\gamma\rho)^{-1}\exp(\rho(\log t)^\gamma)L^2(\exp((\log t)^\gamma))$$ and one may take $g(t,u)=tu^{(\gamma\rho)^{-1}(\log t)^{1-\gamma}}$.

Here is a brief explanation of why the non-standard time scaling $g(t,u)$ appears in Theorem~\ref{main1}. To investigate the joint distribution of $Y(t)$ and $Y(t+g(t,u))$ let us single out the contribution of the points $S_k$ located in the segment $[0, t-g(t,u)]$ by writing
\begin{align*}
Y(t) &= \sum_{k\geq 0} h(t-S_k)\1_{\{S_k\leq t-g(t,u)\}}   + \Delta_1(t),\\
Y(t+g(t,u)) &= \sum_{k\geq 0} h(t + g(t,u)-S_k)\1_{\{S_k\leq t-g(t,u\}} + \Delta_2(t)
\end{align*}
with obvious choices for the remainder terms $\Delta_1(t)$ and $\Delta_2(t)$. Assuming for a moment that $(S_k)$ are the arrival times in a Poisson process of unit intensity, we infer
$$
\Var \left(\sum_{k\geq 0} h(t-S_k)\1_{\{S_k\leq t-g(t,u)\}}\right)
=
\int_{g(t,u)}^{t} h^2(y) {\rm d} y = m(t) - m(g(t,u)) \sim (1-u) m(t)
$$
in view of \eqref{choice}. Similarly, since $m$ is slowly varying which implies $m(t+g(t,u))\sim m(t)$, we obtain
$$
\Var \left(\sum_{k\geq 0} h(t+g(t,u)-S_k)\1_{\{S_k\leq t-g(t,u)\}}\right)
=
\int_{2g(t,u)}^{t+g(t,u)} h^2(y) {\rm d} y \sim (1-u) m(t).
$$
Arguing as above, one can show that the variances of
$\Delta_1(t)$ and $\Delta_2(t)$ are of order $um(t)$, and, moreover, that $\Delta_1(t)$ and $\Delta_2(t)$ are asymptotically independent. Thus, both variables $Y(t)$ and $Y(t+g(t,u))$ have variances of order $m(t)$ and the principal contribution to their covariance (which is asymptotic to $(1-u)m(t)$) comes from the points $S_k$ located in the segment $[0, t-g(t,u)]$. For the renewal shot noise process other than Poisson finding variances, let alone covariances, is a formidable task. Therefore, the argument above should be deemed a useful hint rather than a general approach.

The rest of the paper is structured as follows. In Section \ref{tech} we collect some auxiliary results which are then used in Section \ref{proo} to prove Theorem \ref{main1}.

We stipulate hereafter that all unspecified limit relations hold as $t\to\infty$.

\section{Technical background}\label{tech}

Throughout the section we assume, without further notice, that
$\mu=\E\xi<\infty$.

Let $S^\ast_0$ be a random variable which is independent of
$(S_k)$ and has distribution
$$\Prob\{S^\ast_0\leq x\}=\mu^{-1}\int_0^x \Prob\{\xi>y\}{\rm d}y, \ \ x\geq 0$$ in the case that the distribution of $\xi$ is nonarithmetic, and
$$\Prob\{S^\ast_0=kd\}=(d/\mu)\Prob\{\xi\geq kd\},\quad k\in \N$$ in the case that the distribution of $\xi$ is
arithmetic with span $d>0$. Now we set
\begin{equation}\label{station_rw_def}
S^\ast_k:=S^\ast_0+S_k,\quad k\in\N_0
\end{equation}
and define
$$N(t):=\inf\{k\in\N_0: S_k>t\}=\#\{k\in\N_0: S_k\leq t\},\quad t\geq 0$$ and
$$N^\ast(t):=\#\{k\in\N_0: S^\ast_k\leq t\}, \ \ t\geq
0.$$ Observe that $(N^\ast(x))_{x\geq 0}$ has stationary
increments. It will be important for us that the
finite-dimensional distributions of the increments of
$(N^\ast(x))$ are invariant not only forward in time, but also
backward in time. The latter means that
\begin{equation} \label{eq:N^ backward in time}
(N^\ast(t)-N^\ast((t-s)-): 0 \leq s \leq t)
~\stackrel{\mathrm{f.d.}}{=}~  (N^\ast(s): 0 \leq s \leq t)
\end{equation}
for every $t>0$, see Proposition 3.1 in
\cite{Iksanov+Marynych+Meiners:2014} for more details. Here $\stackrel{\mathrm{f.d.}}{=}$ denotes equality of finite-dimensional distributions. Also, we
have to recall (see p.~55 in \cite{Gut:2009} for the proof) that
$N(t)$ enjoys the following (distributional) subadditivity
property
\begin{equation}\label{yyy}
N(t+s)-N(s)\overset{d}{\leq} N(t), \ \ s,t\geq 0,
\end{equation}
where $X\overset{{\rm d}}{\leq}Y$ means that
$\Prob\{X>z\}\leq\Prob\{Y>z\}$ for all $z\in\R$.

A function $f:\R^+\to\R^+$ is called directly
Riemann integrable (dRi) if (a) $\overline{\sigma}(h)<\infty$ for
some (hence all) $h>0$ and (b) $\lim_{h\to
0+}(\overline{\sigma}(h)-\underline{\sigma}(h))=0$, where
$\overline{\sigma}(h):=\sum_{n\geq 1}\sup_{(n-1)h\leq y<nh}f(y)$
and $\underline{\sigma}(h):=\sum_{n\geq 1}\inf_{(n-1)h\leq
y<nh}f(y)$. A function $f:\R^+\to\R$ is called dRi, if the
functions $f^+:=f\vee 0$ and $f^-:=-(f\wedge 0)$ are dRi (here $x\vee y:=\max(x,y)$). If $f$
is dRi, and the law of $\xi$ is nonarithmetic, then, according to
the key renewal theorem,
$$\lim_{t\to\infty}\E\sum_{k\geq 0}f(t-S_k)\1_{\{S_k\leq
t\}}=\mu^{-1}\int_0^\infty f(y){\rm d}y<\infty.$$ If the law of
$\xi$ is $d$-arithmetic, then this limit relation only holds along
the subsequence $(nd)_{n\in\N}$. In the proof of Theorem \ref{main1} we
want to treat the nonarithmetic and the arithmetic cases
simultaneously. This will be accomplished on using the following
result.
\begin{Lemma}\label{bounded}
If $f:\R^+\to\R^+$ is dRi, then $$\E\sum_{k\geq
0}f(t-S_k)\1_{\{S_k\leq t\}}=O(1),\quad t\to\infty.$$ The same is
true with $S^\ast_k$ replacing $S_k$.
\end{Lemma}
\begin{proof}
The part concerning $S_k$ is Lemma 8.2 in
\cite{Iksanov+Marynych+Vatutin:2015}. The proof of the second part
is analogous.
\end{proof}

The next lemma is a strong approximation result which is one of
the main technical tools in the proof of Theorem \ref{main1}.

\begin{Lemma}\label{appr}
Suppose that $\E\xi^r<\infty$ for some $r>2$. Then there exists a Brownian motion $W$ such that,
for some random, almost surely finite $t_0>0$ and deterministic $A>0$,
$$|N^\ast(t)-\mu^{-1}t-\sigma\mu^{-3/2}W(t)|\leq A t^{1/r}$$ for all $t\geq t_0$, where $\sigma^2=\Var\,\xi$ and $\mu=\E\xi$.
\end{Lemma}
\begin{proof}
According to formula (3.13) in
\cite{Csorgo+Horvath+Steinebach:1987}, there exists a Brownian
motion $W$ such that $$\sup_{0\leq u\leq t}|S_{[u]}-\mu u-\sigma
W(u)|=O(t^{1/r})\quad\text{a.s.}$$ This obviously implies
$$\sup_{0\leq u\leq t}|S^\ast_{[u]}-\mu u-\sigma W(u)|=O(t^{1/r})\quad\text{a.s.}$$ and thereupon $$\sup_{0\leq u\leq
t}|N^\ast(u)-\mu^{-1}u- \sigma\mu^{-3/2}W(u)|=O(t^{1/r})\quad\text{a.s.}$$ by Theorem 3.1 in
\cite{Csorgo+Horvath+Steinebach:1987}. This proves the lemma with possibly random $A$. As
noted in Remark 3.1 of the cited paper the Blumenthal $0-1$ law
ensures that the constant $A$ can be taken deterministic.
\end{proof}

Lemma \ref{bgt1} given below collects two versions of Karamata's
theorem, the results used at least twice in the
paper. Parts (a) and (b) are Proposition 1.5.8 and Proposition
1.5.9a in \cite{Bingham+Goldie+Teugels:1989}, respectively.
\begin{Lemma}\label{bgt1}
Let $r$ be a locally bounded function which varies regularly at $\infty$ with index $\alpha$, i.e., $r(t)=t^\alpha L(t)$ for some $L$ slowly varying at $\infty$.

\noindent (a) If $\alpha>-1$, then $$\int_0^t r(y){\rm
d}y~\sim~(\alpha+1)^{-1}tr(t),\quad t\to\infty.$$

\noindent (b) If $\alpha=-1$, then $t\to \int_0^t r(y){\rm d}y$ is a slowly varying function and
$$\lim_{t\to\infty}{tr(t)\over\int_0^t r(y){\rm d}y}=0.$$
\end{Lemma}

\begin{Lemma}\label{variance}
Let $h$ be a nonincreasing function which satisfies all the assumptions of Theorem \ref{main1}.
Then, for any $0\leq b<a\leq 1$,
$$\lim_{t\to\infty}{\int_0^{t+\ts{b}}h(y)h(y+\ts{a}-\ts{b}){\rm d}y\over \int_0^t h^2(y){\rm d}y}=1-a,$$
where $\ts{u}:=g(t,u)$, $u\in[0,1]$ (see \eqref{choice} for the
definition of $g$).
\end{Lemma}
\begin{proof}
We first treat the principal part of
the integral, namely, we check that
\begin{equation}\label{princ}
\lim_{t\to\infty}{\int_{\ts{a}}^t h(y)h(y+\ts{a}-\ts{b}){\rm d}y\over m(t)}=1-a,
\end{equation}
where the notation $m(t)=\int_0^t h^2(y){\rm d}y$ has to be recalled. We shall frequently use that $\lim_{t\to\infty}\ts{a}/\ts{b}=\infty$ which is a consequence of the slow variation and monotonicity of $m$. By monotonicity of $h$,
$$m(t+\ts{a}-\ts{b})-m(2\ts{a}-\ts{b})\leq \int_{\ts{a}}^t h(y)h(y+\ts{a}-\ts{b}){\rm d}y\leq m(t)-m(\ts{a})$$ which entails \eqref{princ} in view of \eqref{choice} and the slow variation of $m$.
It remains to show that
\begin{equation}\label{limit}
\lim_{t\to\infty}{\int_0^{\ts{a}}h(y)h(y+\ts{a}-\ts{b}){\rm d}y\over
m(t)}=0
\quad\text{and}\quad
\lim_{t\to\infty}{\int_t^{t+\ts{b}}h(y)h(y+\ts{a}-\ts{b}){\rm d}y\over
m(t)}=0.
\end{equation}
As for the first limit  in \eqref{limit}, we have, using again the monotonicity of $h$,
$$
\int_0^{\ts{a}}h(y)h(y+\ts{a}-\ts{b}){\rm d}y \leq h(t^{(a)}- t^{(b)}) \int_0^{t^{(a)}} h(y) {\rm d} y
=
O(\ell^2(t^{(a)}))
$$
by Lemma~\ref{bgt1}(a) since $h$ is regularly varying of index $-1/2$. On the other hand, by Lemma~\ref{bgt1}(b),
$$
\lim_{t\to \infty} \frac{\ell^2(t^{(a)})}{m(t)}
=
\lim_{t\to\infty} \frac{\ell^2(t^{(a)})} {\int_0^{t^{(a)}} y^{-1} \ell^2(y) {\rm d}y} \frac{m(t^{(a)})}{m(t)} = 0
$$
which proves the first limit relation in \eqref{limit}. Turning to the second limit relation,
we have the estimate
$$
\int_t^{t+\ts{b}}h(y)h(y+\ts{a}-\ts{b}){\rm d}y \leq h^2(t) t^{(b)} \sim \ell^2(t) t^{-1} t^{(b)} = o(\ell^2(t)) =  o(m(t)),
$$
where the last step is justified by Lemma~\ref{bgt1}(b). The proof of Lemma \ref{variance} is complete.
\end{proof}

\section{Proof of Theorem \ref{main1}}\label{proo}

The proof consists of several steps. We shall write $\ts{u}$ for $g(t,u)$.

\noindent {\sc Step 1} (Reduction to smooth $h$). The aim is to
show that without loss of generality the function $h$ can be
assumed nonincreasing (everywhere rather than eventually) and
infinitely differentiable with $e^{-t}(-h^\prime(t))$ being
nonincreasing.

By assumption, $h$ is eventually nonincreasing. Hence, there
exists an $a > 0$ such that $h$ is nonincreasing on $[a,\infty)$.
Let $h^\ast$ be a bounded, right-continuous and nonincreasing
function such that $h^\ast(t) = h(t)$ for $t\geq a$. Note that the
so defined $h^\ast$ is nonnegative. The first observation is that
replacing $h$ by $h^\ast$ in the definition of $Y$ does not change
the asymptotics. Indeed\footnote{$Y^\ast$ and $\bar{Y}$ denote the
shot noise processes with the shots occurring at times
$(S_n)_{n\in\N_0}$ and response functions $h^\ast$ and $\bar{h}$
(to be defined below) instead of $h$.}, for large enough $t$,
\begin{eqnarray*}
|Y(t)-Y^\ast (t)|&=&\bigg|\int_{(t-a,\,
t]}(h(t-y)-h^\ast(t-y)){\rm d}N(y)\bigg|\\&\leq& \sup_{y\in
[0,\,a]}|h(y)-h^\ast(y)|(N(t)-N(t-a))\\&\overset{{\rm
d}}{\leq}&\sup_{y\in [0,\,a]}|h(y)-h^\ast(y)|N(a),
\end{eqnarray*}
the last inequality following from \eqref{yyy}. The local
boundedness of $h$ and $h^\ast$ ensures the finiteness of the last
supremum. Further, for large $t$,
$$\bigg|\int_0^t(h(y)-h^\ast(y)){\rm d}y\bigg|\leq \int_0^t|h(y)-h^\ast(y)|{\rm d}y=\int_0^a |h(y)-h^\ast(y)|{\rm
d}y.$$ Since $\lim_{t\to \infty} \int_0^t h^2(y){\rm d}y=\infty$
by the assumption, we have proved that, for any $u\in [0,1]$,
\begin{equation}\label{aux}
{(Y(t+\ts{u})-\mu^{-1}\int_0^{t+\ts{u}}h(y){\rm
d}y)-(Y^\ast(t+\ts{u})-\mu^{-1}\int_0^{t+\ts{u}}h^\ast(y){\rm d}y)\over
\sqrt{\int_0^th^2(y){\rm d}y}}~\stackrel{\mathrm{P}}{\to}~ 0,
\end{equation}
where $\stackrel{\mathrm{P}}{\to}$ denotes convergence in probability.
Replacing $h^\ast(t)$ with $h^\ast(t)/h^\ast(0)$ we can and do
assume that $h^\ast(0)=1$. Then $1-h^\ast(t)$ is the distribution
function of a random variable $|\log W|$, say, where $W\in (0,1)$
a.s.

Set $\bar{h}(t):=\E \exp(-e^tW)$, $t\geq 0$ and observe that the
function $t\to e^{-t}(-\bar{h}^\prime(t))$ is nonincreasing. We
first prove that
\begin{equation}\label{reg}
\bar{h}(t)~\sim~h^\ast(t)~\sim~t^{-1/2}\ell(t).
\end{equation}
By assumption, $h^\ast(t)=\Prob\{|\log
W|>t\}=\Prob\{W<e^{-t}\}\sim t^{-1/2}\ell(t)$ as $t\to\infty$ which entails
$\Prob\{W<t\}\sim |\log t|^{-1/2}\ell(|\log t|)$ as $t\to 0+$.
Hence $\E e^{-tW}\sim (\log t)^{-1/2}\ell(\log t)$ as $t\to\infty$
by Theorem 1.7.1' in \cite{Bingham+Goldie+Teugels:1989}, and
\eqref{reg} follows.

Observe further that
\begin{eqnarray*}
|\bar{h}(t)-h^\ast(t)|&\leq& \E |\exp
(-e^tW)-\1_{\{e^tW<1\}}|\\&=&\E(1-\exp(-e^tW))\1_{\{e^tW<1\}}+\E
\exp(-e^tW)\1_{\{e^tW\geq 1\}}.
\end{eqnarray*} Since, according
to Lemma 8.1 in \cite{Iksanov+Marynych+Vatutin:2015}, the
functions $\E(1-\exp(-e^tW))\1_{\{e^tW<1\}}$ and $\E
\exp(-e^tW)\1_{\{e^tW\geq 1\}}$ are dRi on $\R^+$, so is their sum. This implies that the
function $|\bar{h}(t)-h^\ast(t)|$ is dRi because it is bounded,
continuous and dominated by a dRi function. In particular,
$\int_0^\infty |\bar{h}(y)-h^\ast(y)|{\rm d}y<\infty$ and
furthermore
$$\E\int_{[0,\,t]}|\bar{h}(t-y)-h^\ast(t-y)|{\rm
d}N(y)=O(1)$$ by Lemma \ref{bounded}. Hence, for any $u\in[0,1]$,
$${(\bar{Y}(t+\ts{u})-Y^\ast(t+\ts{u}))-\mu^{-1}(\int_0^{t+\ts{u}}\bar{h}(y){\rm
d}y-\int_0^{t+\ts{u}}h^\ast(y){\rm d}y)\over \sqrt{\int_0^th^2(y){\rm
d}y}}~\stackrel{\mathrm{P}}{\to}~ 0.$$ This in combination with
\eqref{aux} and \eqref{reg} shows that it suffices to prove that
\begin{equation}\label{limit10a}
\bigg({\sum_{k\geq 0}\bar{h}(t+\ts{u}-S_k)\1_{\{S_k\leq
t+\ts{u}\}}-\mu^{-1}\int_0^{t+\ts{u}}\bar{h}(y){\rm d}y \over
\sqrt{\sigma^2\mu^{-3}\int_0^t\bar{h}^2(y){\rm d}y}}\bigg)_{u\in
[0,1]} ~\stackrel{\mathrm{f.d.}}{\to}~ (X(u))_{u\in[0,1]}.
\end{equation}

\noindent {\sc Step 2} (Reduction to renewal
processes with stationary increments). First we  intend to prove that
\begin{equation}\label{tightness1}
{\sum_{k\geq 0}\bar{h}(t-S_k)\1_{\{S_k\leq t\}}-\sum_{k\geq
0}\bar{h}(t-S^\ast_k)\1_{\{S^\ast_k\leq t\}}\over
a(t)}~\stackrel{\mathrm{P}}{\to}~ 0
\end{equation}
for any function $a(t)$ with $\lim_{t\to\infty}a(t)=+\infty$. While doing so we make extensive use of formula \eqref{station_rw_def}.

We start with the equality
\begin{eqnarray*}
\bar{h}(t-S_k)\1_{\{S_k\leq
t\}}&-&\bar{h}(t-S^\ast_k)\1_{\{S^\ast_k\leq
t\}}\\&=&\bar{h}(t-S_k)\1_{\{S_k\leq
t<S^\ast_k\}}-\bigg(\bar{h}(t-S^\ast_k)-\bar{h}(t-S_k)\bigg)\1_{\{S^\ast_k
\leq t\}}
\end{eqnarray*}
which holds a.s. Hence
\begin{eqnarray*}
\bar{h}(t-S_k)\1_{\{S_k\leq
t\}}-\bar{h}(t-S^\ast_k)\1_{\{S^\ast_k\leq t\}}&\leq&
\bar{h}(t-S_k)\1_{\{S_k\leq
t<S^\ast_k\}}\\&=&\bar{h}(t-S_k)\1_{\{S_k\leq t,\,
S^\ast_0>t\}}\\&+&\bar{h}(t-S_k)\1_{\{t-S^\ast_0< S_k\leq t,\,
S^\ast_0\leq t\}} \ \ \text{a.s.}
\end{eqnarray*}
It is clear that $$\bigg(\sum_{k\geq 0}\bar{h}(t-S_k)\1_{\{S_k\leq
t\}}\bigg)\1_{\{S^\ast_0>t\}}~\stackrel{\mathrm{P}}{\to}~ 0.$$
Further, using monotonicity of $\bar{h}$ and \eqref{yyy} gives
$$\sum_{k\geq 0}\bar{h}(t-S_k)\1_{\{t-S^\ast_0< S_k\leq t,\, S^\ast_0\leq
t\}}\leq \bar{h}(0)\big(N(t)-N(t-S^\ast_0)\big)\1_{\{S^\ast_0\leq
t\}}\overset{d}{\leq} \bar{h}(0)N(S^\ast_0).$$ On the other hand,
\begin{eqnarray*}
\bar{h}(t-S_k)\1_{\{S_k\leq
t\}}&-&\bar{h}(t-S^\ast_k)\1_{\{S^\ast_k\leq t\}}\geq
-\bigg(\bar{h}(t-S^\ast_k)-\bar{h}(t-S_k)\bigg)\1_{\{S^\ast_k \leq
t\}}
\end{eqnarray*}
a.s. Invoking now the mean value theorem for differentiable functions
and the fact that $e^{-t}(-\bar{h}^\prime(t))$ is nonincreasing we
obtain, for some $\theta\in [t-S^\ast_k,\,t-S_k]$,
\begin{eqnarray*}
(\bar{h}(t-S^\ast_k)-\bar{h}(t-S_k))\1_{\{S^\ast_k \leq
t\}}&=&e^{-\theta}(-\bar{h}^\prime(\theta))e^\theta
S^\ast_0\1_{\{S^\ast_k \leq
t\}}\\&\leq& e^{-(t-S^\ast_k)}(-\bar{h}^\prime
(t-S^\ast_k))e^{t-S_k}\1_{\{S^\ast_k\leq
t\}}S^\ast_0\\&=&-\bar{h}^\prime (t-S^\ast_k))\1_{\{S^\ast_k\leq
t\}}S^\ast_0e^{S^\ast_0}\quad\text{a.s.}
\end{eqnarray*}
The function $t\to (-\bar{h}^\prime(t))$ is dRi because it is
positive, integrable and the function $t\to
e^{-t}(-\bar{h}^\prime(t))$ is nonincreasing. Hence $\E\sum_{k\geq
0}(-\bar{h}^\prime(t-S^\ast_k))\1_{\{S^\ast_k\leq
t\}}=O(1)$ by Lemma \ref{bounded}. Collecting pieces together we arrive at \eqref{tightness1}. In
view of \eqref{tightness1} relation \eqref{limit10a} is equivalent
to
\begin{equation}\label{limit11}
\bigg({\sum_{k\geq 0}\bar{h}(t+\ts{u}-S^\ast_k)\1_{\{S^\ast_k\leq
t+\ts{u}\}}-\mu^{-1}\int_0^{t+\ts{u}}\bar{h}(y){\rm d}y \over
\sqrt{\sigma^2\mu^{-3}\int_0^t\bar{h}^2(y){\rm d}y}}\bigg)_{u\in
[0,1]} ~\stackrel{\mathrm{f.d.}}{\to}~ (X(u))_{u\in[0,1]}.
\end{equation}

While proving \eqref{limit11} the Cram\'{e}r-Wold device
(see Theorem 29.4 on p.~397 in \cite{Billingsley:1986}) allows us
to work with linear combinations of vector components rather than
with vectors themselves, i.e., it suffices to check that
\begin{equation}\label{limit1}
\sum_{i=1}^n \alpha_i {\sum_{k\geq
0}\bar{h}(t+\ts{u_i}-S^\ast_k)\1_{\{S^\ast_k\leq
t+\ts{u_i}\}}-\mu^{-1}\int_0^{t+\ts{u_i}} \bar{h}(y){\rm d}y\over
\sqrt{\sigma^2\mu^{-3}\int_0^t \bar{h}^2(y){\rm d}y}}
~\stackrel{\mathrm{d}}{\to}~ \sum_{i=1}^n \alpha_i X(u_i)
\end{equation}
for any $n\in\N$, any real $\alpha_1,\ldots, \alpha_n$ and any
$0\leq u_1<\ldots<u_n\leq 1$. Observe that the random variable on
the right-hand side of \eqref{limit1} has the normal distribution
with mean zero and variance $\sum_{i=1}^n\alpha_i^2+2\sum_{1\leq
k<m\leq n}\alpha_k\alpha_m (1-u_m)$.

Integrating by parts we see that the numerator of the left-hand
side of \eqref{limit1} equals
\begin{eqnarray*}
&&\sum_{i=1}^n \alpha_i\int_{[0,\,t+\ts{u_i}]}\bar{h}(t+\ts{u_i}-y)
{\rm d}(N^\ast(y)-\mu^{-1}y)\\&=&
\sum_{i=1}^n\alpha_i\bigg(\bar{h}(t+\ts{u_i})\big(N^\ast(t+\ts{u_i})-\mu^{-1}(t+\ts{u_i})\big)\\&+&
\int_{[0,\,t+\ts{u_i}]}(N^\ast(t+\ts{u_i})-N^\ast((t+\ts{u_i}-y)-)-\mu^{-1}y){\rm
d}(-\bar{h}(y))\bigg).
\end{eqnarray*}
Recall (see \eqref{reg}) that $\bar{h}$ is regularly varying at
$\infty$ of index $-1/2$. Hence
\begin{equation}\label{useful}
\lim_{t\to\infty} {t\bar{h}^2(t)\over \int_0^t\bar{h}^2(y){\rm
d}y}=0
\end{equation}
by Lemma \ref{bgt1}(b) with $r=\bar{h}^2$. Further,
$(N^\ast(t)-\mu^{-1}t)/\sqrt{\sigma^2\mu^{-3}t}$ converges in
distribution\footnote{This follows from the distributional
convergence of $(N(t)-\mu^{-1}(t))/\sqrt{\sigma^2\mu^{-3}t}$ to
the standard normal law (see, for instance, Theorem 5.2 on p.~59
in \cite{Gut:2009}), the representation
$N^\ast(t)=N(t-S_0^\ast)\1_{\{S^\ast_0\leq t\}}$ and
distributional subadditivity \eqref{yyy}.} to the standard normal
law, whence
$$\sum_{i=1}^n \alpha_i{\sqrt{t}\bar{h}(t+\ts{u_i})\over
\sqrt{\int_0^t \bar{h}^2(y){\rm
d}y}}{N^\ast(t+\ts{u_i})-\mu^{-1}(t+\ts{u_i})\over
\sqrt{t}}~\stackrel{\mathrm{P}}{\to}~ 0$$
which shows that
\eqref{limit1} is equivalent to
\begin{equation}\label{limit8a}
\sum_{i=1}^n\alpha_i
{\int_{[0,\,t+\ts{u_i}]}(N^\ast(t+\ts{u_i})-N^\ast((t+\ts{u_i}-y)-)-\mu^{-1}y){\rm
d}(-\bar{h}(y))\over \sqrt{\sigma^2\mu^{-3}\int_0^t
\bar{h}^2(y){\rm d}y}}~\stackrel{\mathrm{d}}{\to}~ \sum_{i=1}^n
\alpha_i X(u_i).
\end{equation}
Reversing the time at the point $t+t^{(u_n)}$ by means of \eqref{eq:N^ backward in time}, we conclude that the left-hand side of \eqref{limit8a} has the same distribution as
\begin{equation*}
\frac{{\sum_{i=1}^n\alpha_i}\int_{[0,\,t+\ts{u_i}]}(N^\ast(y+\ts{u_n}-\ts{u_i})-N^\ast(\ts{u_n}-\ts{u_i})-\mu^{-1}y){\rm
d}(-\bar{h}(y))}{\sqrt{\sigma^2\mu^{-3}\int_0^t
\bar{h}^2(y){\rm d}y}}.
\end{equation*}
Setting $r_m:=\ts{u_n}-\ts{u_{n-m}}$  for  $m=0,\ldots, n-1$ we rewrite \eqref{limit8a} as
\begin{equation}\label{limit8}
\frac{{\sum_{i=1}^n\alpha_i}\int_{[0,\,t+\ts{u_i}]}(N^\ast(y+r_{n-i})-N^\ast(r_{n-i})-\mu^{-1}y){\rm
d}(-\bar{h}(y))}{\sqrt{\sigma^2\mu^{-3}\int_0^t
\bar{h}^2(y){\rm d}y}}\stackrel{\mathrm{d}}{\to} \sum_{i=1}^n
\alpha_i X(u_i).
\end{equation}

\noindent {\sc Step 3} (Reduction to independence). The purpose of the following construction is to replace the increments $(N^{\ast}(y+r_{n-i})-N^{\ast}(r_{n-i}))_{r_{n-i} \leq y \leq r_{n-i+1}}$ (which are dependent) by independent copies of these. Essentially, the overshoots of the random walk $(S_k)_{k\in \N_0}$ at the points $r_{1}, \ldots, r_{n-1}$ are sequentially replaced by independent copies of the random variable $S_0^{\ast}$ while keeping all other increments unchanged.

Let $S_{0,0},\ldots, S_{0,n-1}$ denote independent copies of
$S_0^\ast$ which are also independent of $(\xi_k)_{k\in\N}$.
Further, starting with $$S_k^{\ast(0)}:=S_k^\ast,\quad k\in\N_0,
\quad N^{\ast(0)}(s):=\inf\{k\in\N_0: S_k^{\ast(0)}>s\},\quad
s\geq 0$$ and
$$N^{(0)}(s):=\inf\{k\in\N_0:S_{N^{\ast(0)}(r_1)+k}^{\ast(0)}-S_{N^{\ast(0)}(r_1)}^{\ast(0)}>s\},\quad
s\geq 0$$ we define successively for $m=1,\ldots, n-1$
$$S_k^{\ast(m)}:=r_m+S^{\ast(m-1)}_{N^{\ast(m-1)}(r_m)+k}-S^{\ast(m-1)}_{N^{\ast(m-1)}(r_m)}+S_{0,m},\quad k\in\N_0,$$
$$N^{\ast(m)}(s):=\inf\{k\in\N_0: S_k^{\ast(m)}>s\},\quad s\geq r_m$$ and $$N^{(m)}(s):=\inf\{k\in\N_0: S_{N^{\ast(m-1)}(r_m)+k}^{\ast(m-1)}-
S_{N^{\ast(m-1)}(r_m)}^{\ast(m-1)}>s\},\quad s\geq 0.$$ Observe
that the process $(N^{\ast(m)}(s+r_m))_{s\geq 0}$ is a copy of
$(N^\ast(s))_{s\geq 0}$, and furthermore $(N^{\ast
(m)}(s))_{r_m\leq s\leq r_{m+1}}$ for $m=0,\ldots, n-2$ and
$(N^{\ast(n-1)}(s))_{s\geq r_{n-1}}$ are jointly independent.

The numerator in \eqref{limit8} equals $\sum_{j=1}^n\theta_j+R(t)$,
where
\begin{eqnarray*}
\theta_1&:=&\int_{[0,\,t+\ts{u_1}]}(N^\ast(y+\ts{u_n}-\ts{u_1})-N^\ast(\ts{u_n}-\ts{u_1})\\
&-&\mu^{-1}y){\rm d}\bigg(-\sum_{k=1}^n\alpha_k\bar{h}(y+\ts{u_k}-\ts{u_1})\bigg),
\end{eqnarray*}
\begin{eqnarray*}
\theta_j&:=&\int_{[0,\,\ts{u_j}-\ts{u_{j-1}}]}(N^\ast(y+\ts{u_n}-\ts{u_j})-N^\ast(\ts{u_n}-\ts{u_j})\\
&-&\mu^{-1}y){\rm
d}\bigg(-\sum_{k=j}^n\alpha_k\bar{h}(y+\ts{u_k}-\ts{u_j})\bigg).
\end{eqnarray*}
for $j=2,\ldots, n$ and
\begin{eqnarray*}
R(t)&:=&\sum_{k=2}^{n}\alpha_k \big(N^{\ast}(t^{(u_n)}-t^{(u_1)})-N^{\ast}(t^{(u_n)}-t^{(u_k)})-\mu^{-1}(t^{(u_k)}-t^{(u_1)})\big)\\&\times&\big(\bar{h}(t^{(u_k)}-t^{(u_1)})-\bar{h}(t+t^{(u_k)})\big)\\
&+&\sum_{j=2}^{n-1}\sum_{k=j}^{n}\alpha_k \big(N^{\ast}(t^{(u_n)}-t^{(u_j)})-N^{\ast}(t^{(u_n)}-t^{(u_k)})
-\mu^{-1}(t^{(u_k)}-t^{(u_j)})\big)\\&\times&\big(\bar{h}(t^{(u_k)}-t^{(u_j)})-\bar{h}(t^{(u_k)}-t^{(u_{j-1})})\big)\\&=:&R_1(t)+R_2(t).
\end{eqnarray*}
In view of \eqref{eq:N^ backward in time}, we have
\begin{equation*}
R_1(t)\overset{{\rm d}}{=} \sum_{k=2}^{n}\alpha_k\big(N^{\ast}(t^{(u_k)}-t^{(u_1)})-\mu^{-1}(t^{(u_k)}-t^{(u_1)})\big)\big(\bar{h}(t^{(u_k)}-t^{(u_1)})-\bar{h}(t+t^{(u_k)})\big)
\end{equation*}
and
\begin{eqnarray*}
R_2(t)&\overset{{\rm d}}{=}& \sum_{j=2}^{n-1}\sum_{k=j}^{n}\alpha_k\big(N^{\ast}(t^{(u_k)}-t^{(u_j)})-\mu^{-1}(t^{(u_k)}-t^{(u_j)})\big)\\&\times&
\big(\bar{h}(t^{(u_k)}-t^{(u_j)})-\bar{h}(t^{(u_k)}-t^{(u_{j-1})})\big),
\end{eqnarray*}
whence
$$
\frac{R(t)}{\sqrt{\sigma^2\mu^{-3}\int_0^t
\bar{h}^2(y){\rm d}y}}~\stackrel{\mathrm{P}}{\to}~ 0,\quad t\to\infty
$$
having utilized the central limit theorem for $N^\ast(t)$ (see Step 2), \eqref{useful} and Slutsky's lemma.

Thus, up to a term which tends to zero in probability the numerator in \eqref{limit8} equals the sum $\sum_{k=1}^n\theta_k$ of {\it dependent} random variables. Now we intend to show that instead of this sum we can work with the sum $\sum_{k=1}^n\theta_k^\prime$ of {\it independent} random variables, where $$\theta_1^\prime:=\int_{[0,\,t+t^{(u_1)}]}(N^{\ast(n-1)}(y+\ts{u_n}-\ts{u_1})-\mu^{-1}y){\rm d}\bigg(-\sum_{k=1}^n\alpha_k\bar{h}(y+\ts{u_k}-\ts{u_1})\bigg)$$ and, for $j=2,\ldots, n$,
$$\theta_j^\prime:=\int_{[0,\,\ts{u_j}-\ts{u_{j-1}}]}(N^{\ast(n-j)}(y+\ts{u_n}-\ts{u_j})-\mu^{-1}y){\rm
d}\bigg(-\sum_{k=j}^n\alpha_k\bar{h}(y+\ts{u_k}-\ts{u_j})\bigg).$$

To justify the replacement we shall show that, for $m=1,\ldots, n-1$ and $y\geq 0$,
\begin{equation}\label{ind}
\E|N^\ast(y+r_m)-N^\ast(r_m)-N^{\ast(m)}(y+r_m)|\leq a_m c<\infty,
\end{equation}
where $a_m:=2^m-1$ and $c=2\E N(S_0^\ast)+\E N(y_0)$ for $y_0$ large enough. Note that $c<\infty$ because $\E\xi^2<\infty$ entails $\E S_0^\ast<\infty$ and $\E N(y)\leq \mu^{-1}y+{\rm const}$ for all $y\geq 0$ (Lorden's inequality).

We first prove that, for $m=1,\ldots, n-1$,
\begin{equation}\label{ind1}
I_m:=\E|N^{\ast(m-1)}(y+r_m)-N^{\ast(m-1)}(r_m)-N^{\ast(m)}(y+r_m)|\leq c.
\end{equation}
Indeed,
\begin{eqnarray*}
&&\hspace{-1.5cm}N^{\ast(m-1)}(y+r_m)-N^{\ast(m-1)}(r_m)-N^{\ast(m)}(y+r_m)\\&=&
\sum_{k\geq
0}\1_{\{S_{N^{(m-1)}(r_m)+k}^{\ast(m-1)}-S_{N^{\ast(m-1)}(r_m)}^{\ast(m-1)}\leq
y-(S_{N^{\ast(m-1)}(r_m)}^{\ast(m-1)} -r_m)\}}\\&-&\sum_{k\geq
0}\1_{\{S_{N^{\ast(m-1)}(r_m)+k}^{\ast(m-1)}-S_{N^{\ast(m-1)}(r_m)}^{\ast(m-1)}\leq
y-S_{0,m}\}}\\&=& N^{(m)}(y-\eta_m)\1_{\{\eta_m\leq y\}}-
N^{(m)}(y-S_{0,m})\1_{\{S_{0,m}\leq y\}},
\end{eqnarray*}
where $\eta_m:=S_{N^{\ast(m-1)}(r_m)}^{\ast(m-1)} -r_m$. 
Note that $(N^{(m)}(t))_{t\geq 0}$ is a copy of $(N(t))_{t\geq 0}$
independent of both $\eta_m$ and $S_{0,m}$. The last two random
variables are independent copies of $S_0^\ast$. Further, the
inequality $\E S_0^\ast<\infty$ entails $\lim_{y\to\infty}\E
N(y)\Prob\{S_0^\ast>y\}=0$ because $\E N(y)\sim \mu^{-1}y$ as
$y\to\infty$ by the elementary renewal theorem. With these at hand
we have
\begin{eqnarray*}
I_m&=&\E|N^{(m)}(y-\eta_m)-N^{(m)}(y-S_{0,m})|\1_{\{\eta_m\leq
y,S_{0,m}\leq y\}}\\&+&\E N^{(m)}(y-S_{0,m})\1_{\{\eta_m>y,
S_{0,m}\leq y\}}+\E N^{(m)}(y-\eta_m)\1_{\{\eta_m\leq
y,S_{0,m}>y\}}\\&\leq& \E N^{(m)}(|\eta_m-S_{0,m}|)+2\E
N(y)\Prob\{S_0^\ast>y\}\\&\leq&2\E N(S_0^\ast)+\E N(y_0)
\end{eqnarray*}
for large enough $y_0$, having utilized twice the distributional
subadditivity of $N^{(m)}(t)$ (see \eqref{yyy}) for the first term
on the right-hand side.

To check \eqref{ind} we use mathematical induction. The case $m=1$ has already been settled by \eqref{ind1} (with $m=1$). Suppose \eqref{ind} holds for all $m\leq j-1<n-1$. Then
\begin{eqnarray*}
&&\hspace{-1.5cm}\E|N^\ast(y+r_j)-N^\ast(r_j)-N^{\ast(j)}(y+r_j)|\\&\leq& \E|N^\ast(y+r_j)-N^\ast(r_{j-1})-N^{\ast(j-1)}(y+r_j)|\\&+&\E|N^{\ast(j-1)}(y+r_j)-N^{\ast(j-1)}(r_j)-N^{\ast(j)}(y+r_j)|\\&+&\E |N^{\ast(j-1)}(r_j)+N^\ast(r_{j-1})-N^\ast(r_j)|\leq (2a_{j-1}+1)c=a_jc
\end{eqnarray*}
because the first term does not exceed $a_{j-1}c$ (use \eqref{ind} with $m=j-1$ and $y$ replaced with $y+r_j-r_{j-1}$), the second term does not exceed $c$ (use \eqref{ind1} with $m=j$) and the third term does not exceed $a_{j-1} c$ (use \eqref{ind} with $m=j-1$ and $y=r_j-r_{j-1}$).

Now \eqref{ind} reveals that \eqref{limit8} is equivalent to $${\sum_{k=1}^n \theta_k^\prime \over \sqrt{\sigma^2\mu^{-3}\int_0^t
\bar{h}^2(y){\rm d}y}}~\stackrel{\mathrm{d}}{\to}~ \sum_{i=1}^n
\alpha_i X(u_i).$$

\noindent {\sc Step 4} (Replacing $N^\ast(t)$ with a Brownian
motion). Let $W_0,\ldots, W_{n-1}$ denote independent Brownian motions such that $W_k$ approximates $N^{\ast(k)}(\cdot+\ts{u_n}-\ts{u_{n-k}})$ in the sense\footnote{Recall that $N^{\ast(k)}(\cdot+\ts{u_n}-\ts{u_{n-k}})$ is a renewal process with stationary increments.} of Lemma \ref{appr}.

We claim that
\begin{eqnarray}\label{limit10}
K_{n-1}(t)&:=& \bigg(\int_0^t \bar{h}^2(y){\rm
d}y\bigg)^{-1/2}\int_{[0,\,t+\ts{u_1}]}|N^{\ast(n-1)}(y+\ts{u_n}-\ts{u_1})-\mu^{-1}y\notag\\&-&\sigma\mu^{-3/2}W_{n-1}(y)|{\rm d}\bigg(-\sum_{k=1}^n\alpha_k\bar{h}(y+\ts{u_k}-\ts{u_1})\bigg)~\stackrel{\mathrm{P}}{\to}~ 0
\end{eqnarray}
and that, for $j=2,\ldots, n$,
\begin{eqnarray}\label{limit500}
K_{n-j}(t)&:=& \bigg(\int_0^t \bar{h}^2(y){\rm
d}y\bigg)^{-1/2}\int_{[0,\,\ts{u_j}-\ts{u_{j-1}}]}|N^{\ast(n-j)}(y+\ts{u_n}-\ts{u_j})-\mu^{-1}y\notag\\&-&\sigma\mu^{-3/2}W_{n-j}(y)|{\rm
d}\bigg(-\sum_{k=j}^n\alpha_k\bar{h}(y+\ts{u_k}-\ts{u_j})\bigg)~\stackrel{\mathrm{P}}{\to}~
0.
\end{eqnarray}
With $t_0$ and $A$ as defined in Lemma \ref{appr}, \eqref{limit10}
follows from the inequality
\begin{eqnarray*}
K_{n-1}(t)&\leq& K_{n-1}(t)\1_{\{t_0>t+\ts{u_1}\}}+\bigg(\int_0^t
\bar{h}^2(y){\rm d}y\bigg)^{-1/2}
\\&\times&\bigg(\int_{[0,\,t_0]}|N^{\ast(n-1)}(y+\ts{u_n}-\ts{u_1})-\mu^{-1}y\\&-&\sigma\mu^{-3/2}W_{n-1}(y)|{\rm
d}\bigg(-\sum_{k=1}^n\alpha_k\bar{h}(y+\ts{u_k}-\ts{u_1})\bigg)
\\&+&A \int_{(t_0,\,t+\ts{u_1}]}y^{1/r}
{\rm
d}\bigg(-\sum_{k=1}^n\alpha_k\bar{h}(y+\ts{u_k}-\ts{u_1})\bigg)\bigg)\1_{\{t_0\leq
t+\ts{u_1}\}}
\end{eqnarray*}
because the first two terms on the right-hand side trivially
converge to zero in probability, whereas the third does so, for
the integral $\int_{(t_0,\,\infty)} y^{1/r} {\rm d}(-\bar{h}(y))$
converges (use integration by parts). Relation \eqref{limit500} can be checked along the same lines.

Relations \eqref{limit10} and \eqref{limit500} demonstrate that we reduced the original problem to showing that $${\sum_{k=1}^n \theta_k^{\prime\prime} \over \sqrt{\int_0^t
\bar{h}^2(y){\rm d}y}}~\stackrel{\mathrm{d}}{\to}~ \sum_{i=1}^n
\alpha_i X(u_i),$$ where
$$\theta_1^{\prime\prime}:=\int_{[0,\,t+\ts{u_1}]}W_{n-1}(y){\rm d}\bigg(-\sum_{k=1}^n\alpha_k\bar{h}(y+\ts{u_k}-\ts{u_1})\bigg)$$
and, for $j=2,\ldots, n$,
$$\theta_j^{\prime\prime}:=\int_{[0,\,\ts{u_j}-\ts{u_{j-1}}]}W_{n-j}(y){\rm d}\bigg(-\sum_{k=j}^n\alpha_k\bar{h}(y+\ts{u_k}-\ts{u_j})\bigg).$$
Since $\sum_{k=1}^n \theta_k^{\prime\prime}$ is the sum of {\it independent}
centered {\it Gaussian} random variables it remains to check that $$\Var\,\left(\sum_{k=1}^n\theta_k^{\prime\prime}\right)=\sum_{k=1}^n\Var\,\theta_k^{\prime\prime}~\sim~
\bigg(\sum_{i=1}^n\alpha_i^2+2\sum_{1\leq k<m\leq
n}\alpha_k\alpha_m (1-u_m)\bigg)\int_0^t\bar{h}^2(y){\rm
d}y.$$

Writing the integral defining $\theta_1^{\prime\prime}$ as the limit of
integral sums we infer
\begin{eqnarray*}
\Var\,\theta_1^{\prime\prime}&=&\int_0^{t+\ts{u_1}}\bigg(\sum_{k=1}^n\alpha_k
(\bar{h}(y+\ts{u_k}-\ts{u_1})-\bar{h}(t+\ts{u_k}) \bigg)^2{\rm
d}y\\&=&\sum_{k=1}^n
\alpha_k^2\bigg(\int_{\ts{u_k}-\ts{u_1}}^{t+\ts{u_k}}\bar{h}^2(y){\rm
d}y-2
\bar{h}(t+\ts{u_k})\int_{\ts{u_k}-\ts{u_1}}^{t+\ts{u_k}}\bar{h}(y){\rm
d}y\\&+&(t+\ts{u_1})\bar{h}^2(t+\ts{u_k})\bigg)+2\sum_{1\leq
i<j\leq
n}\alpha_i\alpha_j\bigg(\int_{\ts{u_i}-\ts{u_1}}^{t+\ts{u_i}}\bar{h}(y)\bar{h}(y+\ts{u_j}-\ts{u_i}){\rm
d}y\\&-&
\bar{h}(t+\ts{u_i})\int_{\ts{u_j}-\ts{u_1}}^{t+\ts{u_j}}\bar{h}(y){\rm
d}y-
\bar{h}(t+\ts{u_j})\int_{\ts{u_i}-\ts{u_1}}^{t+\ts{u_i}}\bar{h}(y){\rm
d}y\\&+&(t+\ts{u_1})
\bar{h}(t+\ts{u_i})\bar{h}(t+\ts{u_j})\bigg)\\&=&\sum_{k=1}^n
\alpha_k^2\int_{\ts{u_k}-\ts{u_1}}^{t+\ts{u_k}}\bar{h}^2(y){\rm
d}y\\&+&2\sum_{1\leq i<j\leq n}\alpha_i\alpha_j
\int_{\ts{u_i}-\ts{u_1}}^{t+\ts{u_i}}\bar{h}(y)
\bar{h}(y+\ts{u_j}-\ts{u_i}){\rm d}y+o\bigg(\int_0^t
\bar{h}^2(y){\rm d}y\bigg).
\end{eqnarray*}
The last $o$-term appears because the second, fifth and sixth
terms on the right-hand side of the second equality above are
bounded, whereas the third and seventh terms are $o\bigg(\int_0^t
\bar{h}^2(y){\rm d}y\bigg)$ by \eqref{useful}. Arguing similarly
we obtain, for $m=2,\ldots,n$,
\begin{eqnarray*}
\Var\,\theta_m^{\prime\prime}&=&\int_0^{\ts{u_m}-\ts{u_{m-1}}}\bigg(\sum_{k=m}^n\alpha_k
(\bar{h}(y+\ts{u_k}-\ts{u_m})-
\bar{h}(\ts{u_k}-\ts{u_{m-1}})\bigg)^2{\rm
d}y\\&=&\sum_{k=m}^n\alpha_k^2\int_{\ts{u_k}-\ts{u_m}}^{\ts{u_k}-\ts{u_{m-1}}}\bar{h}^2(y){\rm
d}y\\&+&2\sum_{m\leq i<j\leq
n}\alpha_i\alpha_j\int_{\ts{u_i}-\ts{u_m}}^{\ts{u_i}-\ts{u_{m-1}}}\bar{h}(y)\bar{h}(y+\ts{u_j}-\ts{u_i}){\rm
d}y+o\bigg(\int_0^t\bar{h}^2(y){\rm d}y\bigg).
\end{eqnarray*}
Using these calculations we infer
\begin{eqnarray*}
{\Var\,\big(\sum_{k=1}^n\theta_k^{\prime\prime}\big)\over \int_0^t\bar{h}^2(y){\rm
d}y}&=&\sum_{k=1}^n\alpha_k^2{\int_0^{t+\ts{u_k}}\bar{h}^2(y){\rm
d}y\over \int_0^t \bar{h}^2(y){\rm d}y}\\&+&2\sum_{1\leq i<j\leq
n}\alpha_i\alpha_j
{\int_0^{t+\ts{u_i}}\bar{h}(y)\bar{h}(y+\ts{u_j}-\ts{u_i}){\rm
d}y\over \int_0^t\bar{h}^2(y){\rm d}y}+o(1).
\end{eqnarray*}
As $t\to\infty$, the coefficient of $\alpha_k^2$,
$k=1,\ldots, n$, converges to one. An appeal to Lemma \ref{variance} enables us to conclude that, as $t\to\infty$,
the coefficient of $2\alpha_i\alpha_j$, $1\leq i<j\leq
n$, converges to $1-u_j$. The proof of Theorem \ref{main1} is complete.

\vspace{1cm}
\noindent   {\bf Acknowledgements}  \quad
\footnotesize
We thank the referee for a very careful reading and useful comments which helped improving the presentation. A part of this work was done while A.~Iksanov was visiting M\"{u}nster in
January/February and July 2015. He gratefully acknowledges hospitality and the financial support by DFG SFB 878 "Geometry, Groups and Actions".
The work of A.~Marynych was supported by the Alexander von Humboldt Foundation.

\end{document}